\documentclass[a4paper,12pt,reqno]{amsart}
\usepackage[utf8]{inputenc}
\usepackage{latexsym}\usepackage{ifthen}
\usepackage{enumerate}\usepackage{calc}\usepackage{hyphenat}
\usepackage{amstext,amsbsy,amsopn,amsthm,amsgen,mathabx}
\usepackage{amsfonts,amssymb,amscd,amsxtra,upref}
\usepackage[hmargin=35mm,vmargin=42mm]{geometry}
\usepackage{mathrsfs}\usepackage{euscript}
\usepackage{graphicx,color}\usepackage{verbatim,hyperref}
\usepackage[T1]{fontenc}
\swapnumbers
\theoremstyle{plain}
\makeatletter\@namedef{subjclassname@2020}{\textup{2020} Mathematics Subject Classification}\makeatother
\usepackage{chngcntr}
\newtheorem{Thm}{Theorem}[section]
\newtheorem{Lem}[Thm]{Lemma}
\newtheorem{Cor}[Thm]{Corollary}
\newtheorem{Pro}[Thm]{Proposition}

\theoremstyle{definition}
\newtheorem{Def}[Thm]{Definition}
\newtheorem{Przyk}[Thm]{Example}
\theoremstyle{remark}
\newtheorem{Rem}[Thm]{Remark}
\numberwithin{equation}{section}

\newcommand{\ITEE}[4][]{\ifthenelse{\equal{#2}{#3}}{#4}{#1}}
\ifthenelse{\isundefined{\texorpdfstring}}{\newcommand{\texorpdfstring}[2]{#1}}{}

\newenvironment{abece}[1][\em]{\begin{enumerate}[#1\ (a)]}{\end{enumerate}}
\newenvironment{property}[1]{\begin{equation}%
\label{eq:#1}\begin{minipage}[c]{.85\textwidth}\itshape}{\end{minipage}\end{equation}\ignorespacesafterend}

\newcommand{\rest}[1]{{\raise-.3ex\hbox{\big|}}_{#1}} 
\newcommand{\ZP}{{\mathbb{Z}_+}}
\newcommand{\C}{\mathbb{C}}
\newcommand{\N}{\mathbb{N}}

\newcommand{\pImp}[2]{\par(\ifthenelse{\equal{#2}{#1}}%
{$\Rightarrow$}{#1)$\Rightarrow$(#2})}
\newcommand{\lImp}[2]{\par(\ifthenelse{\equal{#2}{#1}}%
{$\Leftarrow$}{#1)$\Leftarrow$(#2})}
\newcommand{\NWSR}{the following conditions are equivalent:}
\newcommand{\scp}[1]{\left\langle#1\right\rangle}

\newcommand{\dom}{\mathscr{D}}
\newcommand{\eld}[1]{\ell^2(#1)}
\newcommand{\Tree}{\mathcal{T}}\newcommand{\Troot}{\mathsf{root}}
\newcommand{\parf}{\mathsf{p}}\newcommand{\parq}{\mathsf{q}}\newcommand{\Des}{\mathsf{Des}}
\newcommand{\Chif}[1][]{\mathsf{Chi}\ifthenelse{\equal{#1}{}}{}{^{\langle #1\rangle}\!}}
\newcommand{\Chit}[2]{\mathsf{Chi}^{\langle #1\rangle}_{#2}}

\newcommand{\blambda}{{\boldsymbol\lambda}}

\allowdisplaybreaks[3]
\begin{document}
\title[Weighted shifts on directed forests and hyponormality]
{Weighted shifts on directed forests\\and hyponormality} 
\author[P.\ Pikul]{Piotr Pikul}
\address{P. Pikul\\Instytut Matematyki\\
Wydzia\l{} Matematyki i~Informatyki\\Uniwersytet Jagiello\'{n}ski\\
ul.\ \L{}ojasiewicza 6\\30-348 Krak\'{o}w\\Poland}
\email{Piotr.Pikul@im.uj.edu.pl}
\begin{abstract}
In a 2012 paper, Jab\l o\'nski, Jung and Stochel introduced the
weighted shifts on directed trees, a generalisation of the well-known
weighted shift operators on $\ell^2$. In the last decade this class
has proven itself handy for finding  counterexamples in operator
theory. Properties of the underlying graph structure had essential
influence on these operators.
It appears that a slight generalisation of the class, namely weighted
shifts on directed \emph{forests}, shows even deeper relations between
graph theory and operator theory. Several operations on directed
forests have their natural operator-theoretic counterparts.
This paper is meant to present advantages of the directed forest
approach. As an application of the interrelation between graphs and
operators we provide a full characterisation of directed forests on
which every hyponormal bounded weighted shift is power hyponormal.
\end{abstract}
\maketitle

\section{Introduction} 
The weighted shifts on $\ell^2$ are long known objects in operator
theory (see e.g.\ \cite{shields} for a survey).
We recall that a weighted shift operator $W$ on $\ell^2$ is determined
by the formula $We_n = a_n e_{n+1}$, where $\{ a_n \}_{n=0}^\infty$ is
a~sequence of complex weights\footnote{According to \cite[Corollary 1]{shields}
(cf.\ \cite[Theorem~3.7]{pikul1}), there is no loss of generality in
assuming that the weights are non-negative.} and $\{e_n\}_{n=0}^\infty$ is
the standard orthonormal basis of $\ell^2$.

In \cite{szifty} there was introduced a generalisation of this concept,
namely a weighted shift on a directed tree.
To state the idea briefly, dealing with trees we map (shift) an element of
the orthonormal basis onto more than one ``successor'' and multiply them by weights. 
The formula defining the action of such operator (in the bounded case) 
takes the form (cf.\ \cite[Proposition 3.3]{pikul1})
\begin{equation}\label{eq:naBazowym}
 S_\blambda e_v = \sum_{u\in\Chif(v)} \lambda_u e_u,
\end{equation}
where $\Chif(v)$ is the set of children of the vertex $v$.
This class of operators helped to solve several open problems
in operator theory (see, e.g., \cite{anand,budz2,exner,jabjs1,jabjs2}).

Several generalizations of weighted shifts have been studied in the literature, 
for example multivariable weighted shifts \cite{curto1k,back-ext-2D}.

The concept of weighted shifts on directed forests we are going to
focus on in this paper was introduced in \cite{pikul1}.
This seemingly minor change, allowing non-connected graphs,
appears to show even tighter connection between operator-theoretic
properties and graph structure. One of the advantages is
that the class of weighted shifts on directed forests
is closed on taking powers or orthogonal sums. Moreover,
these operations on operators have their counterparts in
the underlying graph structure.

Many properties and arguments are similar to those of directed
trees presented in \cite{szifty}.
Up to a unitary equivalence and a slight modification
of the underlying forest, all weights\footnote{Except for the weights
attached to the roots, which equal $0$ by definition.}
are strictly positive (\cite[Theorem 3.7]{pikul1}). 
Any weighted shift on a directed forest is a circular
operator (see \cite{circular}), and hence its spectrum
is circular (i.e.\ rotationally invariant).
In the case of a weighted shift on a directed tree, a zero weight
decomposes it to an orthogonal sum (see \cite[Proposition 3.1.6]{szifty}).

An even wider class of graphs and associated operators could
be studied. Research on this subject has a rich history
(see, e.g., \cite{mohar,muller,fujii,majdak,adj-ops}).

For directed forests several ``graph operations'' can be defined,
like taking powers (see \cite[Definition 2.12]{pikul1}) 
of disjoint unions. Also the relation of \emph{thickness}
(see Definition~\ref{dfn:thickness}) is a useful tool when dealing
with forests. In \cite{pikul1,pikul2} some additional operations
were defined, particularly for the study of backward extensions.

In the case of classical weighted shifts there is no difference
between hyponormality and power hyponormality.
It is well known that a wide class of composition operators on $L^2$
spaces shares this property (see \cite[Corollary 3]{campbell}).
However, as shown in \cite[Example]{campbell} there exist hyponormal
composition operators whose squares are not hyponormal (see
\cite{halmosPB,ito} for more information on squares of hyponormal
operators). 
As a final result we provide a full characterisation of all
directed forests on which bounded hyponormal weighted shifts are power hyponormal
(see Theorems~\ref{thm:only-stars} and \ref{thm:star-support}).
According to this characterisation, any directed forest whose
leafless support is non-forkless admits a hyponormal weighted shift
whose square is not hyponormal
(cf.\ \cite[Examples 5.3.2 and 5.3.3]{szifty}).
For the relationship to the Halmos problem on polynomial hyponormality,
we refer the reader to \cite{curto-putinar,curto-putinar2} (see also
\cite{coul-paulsen,curtoq} for the case of classical weighted shifts).

One may notice that passing from the previous result \cite[Theroem~4.8]{pikul1}
to the complete characterisation (Theorem~\ref{thm:star-support})
is essentially done by studying relations between directed forests.
This is a good summary of how strong is the bond between graph structure 
and properties of operators. 

\section{Weighted shifts on directed forests}\label{R:lasy}
We denote the sets of non-negative integers, positive integers,
and complex numbers by $\ZP$, $\N$ 
 and $\C$, respectively.
If $f\colon A\to A$ is any function, then $f^0=\mathrm{id}_A$
(the identity function on $A$).

This section recalls the foundations of weighted shifts on directed forests.
introduced in \cite{pikul1} (cf.\ also \cite{pikul2}).
Before we start discussing weighted shift operators we
take a look on the directed graphs lying underneath.

The first thing worth noting is that we do not follow the traditional
formal framework of graph theory. In fact the possibility of
using much more concise definitions and notation is one of the advantages
of directed forests over trees.

\begin{Def}[{\cite[Definition 2.1]{pikul1}}]\label{dfn:dforest}
A pair $\Tree=(V,\parf)$ is called a~\emph{directed forest}
if $V$ is a~nonempty set and $\parf\colon V\to V$
satisfies the following condition:
\begin{property}{fdef}
if $n\in\N$, $v\in V$ and $\parf^n(v)=v$, then $\parf(v)=v$.
\end{property}
Elements of the set $V$ are called \emph{vertices} of the forest
and $\parf$ is called the \emph{parent function}. We introduce
also the set of \emph{roots}:
 $\Troot(\Tree):=\{v\in V\colon \parf(v)=v\}$. 
\end{Def}

Unless otherwise stated, $\Tree=(V,\parf)$ denotes
a directed forest.
From the operator-theoretic point of view infinite forests
(i.e\ with infinitely many vertices) are of higher interest.

For $v\in V$, we define the set of \emph{$k$-th children of $v$} as
\begin{equation}\label{eq:child-sets}
\Chif[0](v):=\{v\},\quad \Chif[k](v):=\{u\in V\colon \parf^k(u)=v\neq \parf^{k-1}(u)\},\quad k\in\N.
\end{equation}
We skip the index $k=1$ and call elements of the set $\Chif(v)=\Chif[1](v)$
simply \emph{children of $v$}. If $u\in\Chif(v)$, then $v$ is called the \emph{parent} of $u$.
The cardinality of the set $\Chif(v)$ is called \emph{the degree}
of the vertex $v$ and denoted by $\deg(v)$.
We also denote the set of all \emph{descendants} of a vertex $v$
by $\Des(v):=\bigcup_{n=0}^\infty \Chif[n](v)$.
A handy observation (see \cite[Lemma 2.5]{pikul1}) is that
\begin{equation}\label{eq:Des}
\Des(v)=\{u\in V \colon \exists_{n\in\ZP}\ \parf^n(u)=v\,\}.
\end{equation}

If there is a risk of ambiguity, we write $\Chif_\Tree(v)$, $\Des_\Tree(v)$,
etc.\ to make the dependence on $\Tree$ explicit.

A \emph{leaf}\footnote{Note that if a root has no children (it is not considered
as a child of itself) it still cannot be called a `leaf'.} is an element
of $V\setminus \parf(V)$. 
If there is no leaf (i.e.\ $\parf(V)=V$) then the forest is called \emph{leafless}.
We also set $V^\circ:=V\setminus \Troot(\Tree)$. 
The forest is called \emph{degenerate} if $V^\circ = \emptyset$
(i.e.\ $\parf=\mathrm{id}_V$). Otherwise it is non-degenerate.
Basic elements of a directed forest are presented in Figure~\ref{fig:basic}.

\begin{figure}[hbt]
\includegraphics[width=.75\textwidth]{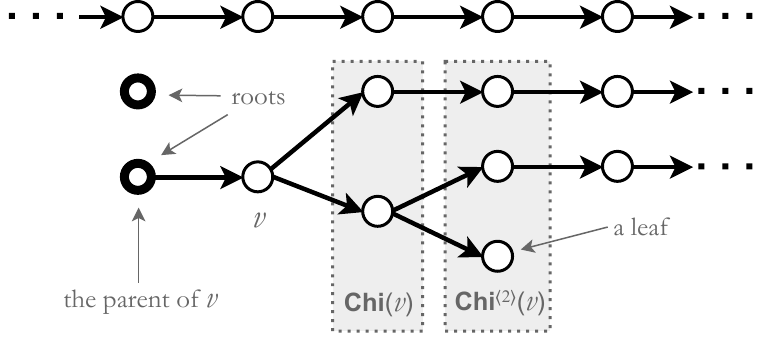}
\caption{An example of a directed forest.
The arrows lead from $\parf(u)$ to $u$ for $u\in V^\circ$.}
\label{fig:basic}\end{figure}

\begin{Rem}
If we skip the condition \eqref{eq:fdef} in Definition~\ref{dfn:dforest}
we would obtain a graph called \emph{forested circuit} \cite[Section 6]{adj-ops}.
However, then the property \cite[Lemma~2.2 (b)]{pikul2} 
no longer holds (with $N^+_{\langle k\rangle}(v)=\Chif[k](v)$).
\end{Rem}

Some basic facts regarding elements of a directed forest
were listed in \cite[Lemma 2.4]{pikul1}. Below we provide
one more handy property, the elementary proof of which is omitted.
\begin{Lem}\label{lem:tree-basics}
Let $\Tree=(V,\parf)$ be a directed forest.
If $\parf^n(v)\in V^\circ$ for some $n\in\ZP$, then
$\parf^k(v)\neq\parf^l(v)$ provided that $0\leq k,l\leq n+1$ and $k\neq l$,
\end{Lem}

\begin{Cor}\label{cor:leafless}
Let $\Tree=(V,\parf)$ be a directed forest. Then $\Tree$ is leafless
if and only if for each $v_0\in V^\circ$ there exists a sequence
$\{v_n\}^{\infty}_{n=1}\subseteq V^\circ$ such that $\parf(v_{n+1})=v_n$
for all $n\in\ZP$.
\end{Cor}
\begin{proof}
If $\Tree$ is leafless, then the required sequence can be constructed
by induction using (f) and (d) from \cite[Lemma 2.4]{pikul1}.

If the sequence $\{v_n\}^{\infty}_{n=0}\subseteq V^\circ$ such that $\parf(v_{n+1})=v_n$
for all $n\in\ZP$ is given, then by 
Lemma~\ref{lem:tree-basics} $v_0\neq v_1$ and hence
$v_1\in\Chif(v_0)$ and $v_0$ is not a leaf.
\end{proof}

By a \emph{tree} in the directed forest $\Tree$ we mean
a connected component of the graph (cf.\ \cite[page~5]{pikul1}).
For a vertex $v\in V$ in a directed forest we will denote the tree
to which $v$ belongs by $[v]=[v]_\Tree \subseteq V$.
If $[v]=\{v\}$, then we call such a tree \emph{degenerate}.
Every tree in a degenerate forest is degenerate.
Clearly, $\parf^n(v)\in[v]$ and $\Chif[n](v)\subseteq[v]$
for every $v\in V$ and $n\in\ZP$.
Basic properties of trees are listed in \cite[Lemma 2.5]{pikul1}.

The directed forest whose vertices are all of degree at most
$\aleph_0$ (equivalently: all trees are countable)
will be called \emph{locally countable}.
A~\emph{directed tree} is a directed forest containing only one tree (a ``connected forest'').
Our definition of a directed tree is consistent with the ``traditional'' one
(see, e.g., \cite{szifty}), as described by \cite[Remark~2.3]{pikul2}.

For directed forests a natural notion of isomorphism can be introduced
(see \cite[Definition 2.8]{pikul1}).
In most cases we are not making any distinction between isomorphic forests.
Forests are said to be disjoint if their sets of vertices are disjoint.
We can always assume that the considered forests are disjoint by
taking appropriate isomorphic forests if necessary.


For a~set $V$ we consider Hilbert space $\eld{V}$ consisting of all
square-summable complex functions on $V$ with the following inner product
\[ \scp{f,g}:=\sum_{v\in V} f(v)\overline{g(v)},\quad f,g\in\eld{V}.\]
The canonical orthonormal basis of $\eld{V}$ will be denoted by $\{e_v:v\in V\}$.

\begin{Def}[{\cite[Definition 3.1]{pikul1}, see also \cite[Definition 3.1.1]{szifty}}]\label{dfn:shift-op}
For a~directed forest $\Tree=(V,\parf)$ and a~set
of complex weights $\blambda=\{\lambda_v\}_{v\in V}$ such that $\lambda_\omega=0$ for
$\omega\in \Troot(\Tree)$,
we define the operator $S_\blambda$ in $\eld{V}$, called the
\emph{weighted shift on $\Tree$ with weights $\blambda$}, as follows:
\[ \dom(S_\blambda):=\{f\in \ell^2(V)\colon \varGamma_\blambda(f) \in \eld{V}\}, \]
\[ S_\blambda (f):= \varGamma_\blambda(f),\quad f\in \dom(S_\blambda),\]
where $\varGamma_\blambda \colon\C^V\to \C^V$ is defined by
$\varGamma_\blambda(f)(v):= \lambda_v f(\parf(v))$.
In what follows, saying ``$\blambda$ is a system of weights on
$\Tree$'', we always assume that $\lambda_\omega=0$ for $\omega\in\Troot(\Tree)$.

We call a weighted shift $S_\blambda$ \emph{proper} if
$\{v\in V\colon\lambda_v=0\}= \Troot(\Tree)$.
\end{Def}

In this article we focus on bounded operators exclusively. 
For a weighted shift this is equivalent to the finiteness of
$\sup\{\sum_{u\in\Chif(v)} |\lambda_u|^2: v\in V\}$ and then the action
of the operator is described by the formula \eqref{eq:naBazowym}.
This and other basic properties of weighted shifts on directed forests were
presented in \cite[Proposition 3.3]{pikul1}. They do not differ
from the case of directed trees described in \cite{szifty}
(cf.\ Propositions 3.1.3 and 3.1.8 therein).

In view of Definition~\ref{dfn:shift-op} (cf.\ formula \eqref{eq:naBazowym}),
a classical unilateral weighted shift on $\eld{\ZP}$ is a weighted shift on
the directed tree $\Tree=(\ZP,\parf)$, where $\parf(0)=0$ and $\parf(n+1)=n$
for $n\geq 0$. For a bilateral weighted shift (acting on $\eld{\mathbb Z}$)
we use the tree $\Tree=(\mathbb Z,\parf_{\mathbb Z})$ with
$\parf_{\mathbb Z}(n)=n-1$ for $n\in\mathbb Z$\label{dfn:linear-tree}.


Many properties extend from weighted shifts on directed trees
to the weighted shifts on directed forests.
For example, the description of the polar decomposition
\cite[Proposition 3.5.1]{szifty} or criteria for hyponormality
(\cite[Theorem 5.1.2]{szifty}, cf.\ \cite[Theorem 6.2]{pikul1})
and subnormality (\cite[Theorem 6.1.3]{szifty},
cf.\ \cite[Theorem 5.3]{pikul1}) remain valid.

\section{Relation of thickness}
Apart from isomorphism, there is another important relation on
directed forests. It is the major tool introduced in this article
which was not present in \cite{pikul1}.
\newcommand{\treeleq}{\llcurly}
\begin{Def}[Relation of thickness] \label{dfn:thickness}
Given $\parf_1,\parf_2:V\to V$ we write
$\parf_1\treeleq \parf_2$ if $\parf_1(v)\in\{v,\parf_2(v)\}$
for every $v\in V$.
If $\Tree_1=(V,\parf_1)$ and $\Tree_2=(V,\parf_2)$ are
directed forests with a common set of vertices, then $\Tree_1$ is
\emph{thinner} than $\Tree_2$
(or $\Tree_2$ is \emph{thicker} than $\Tree_1$)
if $\parf_1\treeleq \parf_2$.
We write $\Tree_1 \treeleq \Tree_2$.
\end{Def}

In a graph theoretic sense, one forest is thinner than another
if its set of edges is smaller in the sense of inclusion
(cf.\ \cite[Remark~2.3]{pikul2} 
and Theorem~\ref{thm:thickness}~\eqref{thic:child1}).
See Figure~\ref{fig:thickness} for illustration.
\begin{figure}[hbt]\vspace{1.5ex}
\includegraphics{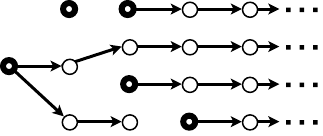}\qquad\includegraphics{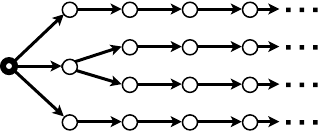}
\caption{Two directed forests with the left one being thinner
than the right one. Bold circles mark the roots.}\label{fig:thickness}
\end{figure}

\begin{Pro}\label{pro:thic-order}
For any set $V$, the relation $\treeleq$ is a partial order
on the set of all self maps of $V$. Moreover, the smallest
element is the identity map.
\end{Pro}
\begin{proof}
The fact that $\mathrm{id}_V\treeleq \parf$ for any
$\parf:V\to V$ is evident, since always $\mathrm{id}_V(v)\in\{v,\parf(v)\}$. 

\emph{Reflexivity.} Clearly $\parf(v)\in\{v,\parf(v)\}$, hence $\parf\treeleq \parf$.

\emph{Antisymmetry.} Assume that $\parf_1(v)\neq\parf_2(v)$ for
some $v\in V$. If $\parf_1\treeleq \parf_2$,
then $\parf_1(v)\in\{v,\parf_2(v)\}\setminus\{\parf_2(v)\}$
implying that $\parf_1(v)=v$. Hence $\parf_2(v)\notin\{v,\parf_1(v)\}$
and $\parf_2\treeleq \parf_1$ cannot hold. In particular
we obtain, that if $\parf_1\treeleq\parf_2$ and
$\parf_2\treeleq\parf_1$, then $\parf_1=\parf_2$.

\emph{Transitivity.}
Let $\parf_1,\parf_2,\parf_3$ be self maps of $V$ satisfying
$\parf_1\treeleq \parf_2$ and $\parf_2\treeleq \parf_3$.
Pick $v\in V$. By definition, $\parf_1(v)\in\{v,\parf_2(v)\}$.
Since $\parf_2(v)\in\{v,\parf_3(v)\}$, we obtain that
$\parf_1(v)\in\{v,\parf_3(v)\}$ and the proof is complete.
\end{proof}

The following two lemmata will be useful in the future proofs.

\begin{Lem}\label{lem:thin-forest}
Let $\Tree=(V,\parf_0)$ be a directed forest and $\parf_1:V\to V$
be such that $\parf_1\treeleq\parf_0$.
Then $(V,\parf_1)$ is a directed forest.
\end{Lem}
\begin{proof}
It follows from \cite[Lemma 2.6]{pikul1} 
applied to
$\Tree_0=\Tree$ and $R=\{v\in V: \parf_0(v)\neq\parf_1(v)\}$.
Observe that $\parf_1\treeleq\parf_0$ implies that $\parf=\parf_1$.
\end{proof}

\begin{Lem}\label{lem:thin-anc}
Let $\Tree_1=(V,\parf_1)$ and $\Tree_2=(V,\parf_2)$ be directed
forests satisfying $\Tree_1 \treeleq \Tree_2$. Fix $v\in V$ and
$n\in\ZP$. Then either $\parf_1^n(v)=\parf_2^n(v)$ or
$\parf_1^k(v)=\parf_2^k(v) \in\Troot(\Tree_1)$ for some $0\leq k< n$.
\end{Lem}
\begin{proof} 
Fix $n\in\ZP$. If $\parf^n_1(v)\neq\parf^n_2(v)$, then we define
\[ k=\max\{m\in\ZP\colon \parf^j_1(v)=\parf^j_2(v),\ j=0,\ldots,m\}. \]
Clearly, $k<n$. From $\Tree_1\treeleq\Tree_2$ we infer that either 
$\parf_1^{k+1}(v)=\parf_2(\parf_1^k(v))$ or
$\parf_1^{k+1}(v)=\parf_1^k(v)$, and consequently $\parf_1^k(v)\in\Troot(\Tree_1)$.
In the first case, we see that $\parf_1^{k+1}(v)=\parf_2^{k+1}(v)$, which contradicts
maximality of $k$. Hence, we have $\parf_1^k(v)=\parf_2^k(v)\in\Troot(\Tree_1)$.
This completes the proof.
\end{proof}

\begin{Thm}[Properties of the thickness relation]\mbox{}
\label{thm:thickness}\begin{abece}
\item \label{thic:child1}
If $\Tree_1=(V,\parf_1)$ and $\Tree_2=(V,\parf_2)$ are directed forests,
then $\Tree_1 \treeleq \Tree_2$ if and only if
\begin{equation}
\Chif_{\Tree_1}(v)\subseteq \Chif_{\Tree_2}(v) \text{ for every } v\in V.
\label{eq:thick-chi}
\end{equation}
\item\label{thic:order}
Thickness is a partial order on the set of all directed forests
with a fixed set of vertices.
\item\label{thic:minimum}
The degenerate forest is the thinnest directed forest with a fixed
set of vertices $($the smallest element with respect to the partial order of thickness$)$.
\item\label{thic:maximum}
The maximal elements with respect to the relation of thickness are
precisely the directed trees and rootless directed forests.
\end{abece}\end{Thm}
\begin{proof}
\eqref{thic:child1}
`$\Rightarrow$' Pick $v\in V$ and $u\in \Chif_{\Tree_1}(v)$. Since
$v=\parf_1(u)\in\{u,\parf_2(u)\}$ and $u\neq v$ (cf.~\eqref{eq:child-sets})
we obtain $\parf_2(u)=v$, and consequently $u\in\Chif_{\Tree_2}(v)$.

`$\Leftarrow$' Consider $v\in V$. If $v\in\Troot(\Tree_1)$,
then clearly $\parf_1(v)\in\{v,\parf_2(v)\}$.
For $v\in V\setminus\Troot(\Tree_1)$, we have
$v\in\Chif_{\Tree_1}(u) \subseteq \Chif_{\Tree_2}(u)$, where
$u:=\parf_1(v)\neq u$. Hence $\parf_2(v) = u=\parf_1(v)$ and
thus $\parf_1(v)\in\{v,\parf_2(v)\}$.

\eqref{thic:order} and \eqref{thic:minimum}
are covered by Proposition~\ref{pro:thic-order}.

\eqref{thic:maximum}
Assume $\Tree_0=(V,\parf_0)$ is a rootless directed forest and
$\Tree_0 \treeleq \Tree_1$ for some directed forest $\Tree_1=(V,\parf_1)$.
For any $v\in V$, $\parf_0(v) \in \{v, \parf_1(v)\}$ implies
$v\neq \parf_0(v)=\parf_1(v)$ and proves maximality of $\Tree_0$.

Now consider a rooted directed tree (the case of rootless directed tree
was already solved) $\Tree_1=(V,\parf_1)$ and assume
$\Tree_1 \treeleq \Tree_2 =(V,\parf_2)$. Let $\omega$ be the root
of $\Tree_1$. For $v\in V\setminus \{\omega\}$ we repeat the argument above
obtaining $\parf_1(v)=\parf_2(v)$ for $v\neq \omega$. Now we will prove that
$\parf_2(\omega)=\parf_1(\omega)$.
Denote $v:=\parf_2(\omega)$. Then, since $V=\Des_{\Tree_1}(\omega)$, there exists minimal
$n\in\ZP$ such that $\parf_1^n(v)=\omega$.
According to Lemma~\ref{lem:thin-anc}, either $\parf_2^n(v)=\omega$, or
$\parf_1^k(v)\in\Troot(\Tree_1)=\{\omega\}$ for some $k<n$ which contradicts
the minimality of $n$. Hence $\parf_2^{n+1}(\omega)=\omega$,
which by \eqref{eq:fdef} gives $\parf_2(\omega)=\omega =\parf_1(\omega)$,
and consequently $\Tree_1=\Tree_2$.

To complete the proof of \eqref{thic:maximum} pick a~directed forest
$\Tree_1=(V,\parf_1)$ which is not rootless and consists of more than one tree.
We will find a~directed forest $\Tree_2\neq \Tree_1$ such that 
$\Tree_1\treeleq\Tree_2$.
Pick $\omega\in\Troot(\Tree_1)$ and $v_0\in V\setminus[\omega]_{\Tree_1}$
(we recall that $[\omega]_{\Tree_1}$ stands for the tree containing $\omega$).
Set $\parf_2(v):=\parf_1(v)$ for $v\in V\setminus\{\omega\}$ and $\parf_2(\omega):=v_0$ (cf.\ Figure~\ref{fig:thick-max}).
To prove that a directed forest $\Tree_2:=(V,\parf_2)$ is well defined
assume that $\parf_2^n(v)= v$ for some $n\geq 1$.
If $v\notin [\omega]_{\Tree_1}$, we have $\parf_2^k(v)=\parf_1^k(v)$ for any
$k\in\ZP$ (trees are invariant sets for $\parf_1$) and by \eqref{eq:fdef}
with $\parf=\parf_1$ we obtain $\parf_2(v)=\parf_1(v)=v$.
Now assume that $v\in  [\omega]_{\Tree_1} = \Des_{\Tree_1}(\omega)$.
If $\parf_1^k(v)\neq \omega$ for every $k=0,\ldots,n-1$, then
$v=\parf_2^n(v)=\parf_1^n(v)$ which implies that $v\in\Troot(\Tree_1)$,
a contradiction.
Otherwise, if $\parf_1^{n_0}(v)= \omega$ for some $n_0<n$,
then $\parf_2^n(v)=\parf_2^{n-n_0}(\omega)\notin [\omega]_{\Tree_1}$,
which contradicts $v=\parf_2^n(v)$.
By definition $\Tree_2\neq \Tree_1$ and $\Tree_1\treeleq\Tree_2$,
which completes the proof.
\end{proof}
\begin{figure}[hbt]
\includegraphics[scale=1]{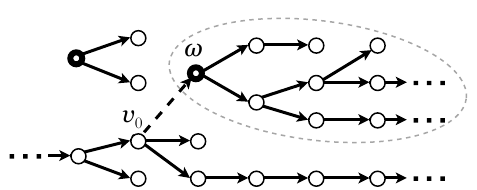}
\caption[Construction of a strictly thicker directed forest]%
{Construction of a strictly thicker directed forest. See Theorem~\ref{thm:thickness}~\eqref{thic:maximum}.}
\label{fig:thick-max}
\end{figure}

As one might expect, the relation of thickness implies a comparison of
several structural elements of the forests involved.
\begin{Lem}\label{lem:thick2}
Let $\Tree_1=(V,\parf_1)$ and $\Tree_2=(V,\parf_2)$ be directed forests
satisfying $\Tree_1 \treeleq \Tree_2$. Then
\begin{abece}
\item \label{thic:childn} 
$\displaystyle \Chit{k}{\Tree_1}(v)\subseteq \Chit{k}{\Tree_2}(v)$
for all $v\in V$ and $k\in\N$,
\item\label{thic:roots}
$\Troot(\Tree_1)\supseteq \Troot(\Tree_2)$,
\item\label{thic:trees}
every tree in $\Tree_1$ is contained in some tree in $\Tree_2$;
in particular, the number of trees in the forest $\Tree_2$ is less than or equal
to the number of trees in $\Tree_1$,
\item\label{thic:trees-eq}
every tree in $\Tree_2$ contains exactly one tree in $\Tree_1$
if and only if $\Tree_1=\Tree_2$; in particular, a directed forest strictly
thinner than a directed tree is never a directed tree.
\end{abece}
\end{Lem}
\begin{proof}
\eqref{thic:childn} The inclusions can be proven by induction using \eqref{eq:thick-chi}
and \cite[Lemma 2.4 (c)]{pikul1}. 

\eqref{thic:roots} If $\omega\in\Troot(\Tree_2)$, then $\parf_1(\omega)\in\{\omega,\parf_2(\omega)\}=\{\omega\}$
and hence $\omega$ is a root of $\Tree_1$.

\eqref{thic:trees}
From $\eqref{thic:childn}$ one can deduce that
$\Des_{\Tree_1}(v)\subseteq\Des_{\Tree_2}(v)$ for $v\in V$.
Fix $v_0\in V$. We will show that
\begin{equation}\label{eq:thin-anc}
\{\parf_1^n(v_0)\colon n\in \ZP\}\subseteq\{\parf_2^n(v_0)\colon n\in \ZP\}.
\end{equation}
Then, applying \cite[Lemma 2.5 (a)]{pikul1}, 
we will obtain that the tree in $\Tree_1$ containing $v_0$
is a subset of the tree in $\Tree_2$ containing $v_0$.

Now we prove \eqref{eq:thin-anc}. According to Lemma~\ref{lem:thin-anc},
either $\parf_1^n(v_0)= \parf_2^n(v_0)$ for every $n\in\ZP$ and we are done,
or $\parf_1^k(v_0)=\parf_2^k(v_0)\in \Troot(\Tree_1)$ for some $k\in\ZP$.
In the latter case, picking $k$ to be minimal, we get 
\[ \{\parf_1^n(v_0)\}_{n=0}^\infty= \{\parf_1^n(v_0)\}_{n=0}^k = \{\parf_2^n(v_0)\}_{n=0}^k, \]
so the inclusion \eqref{eq:thin-anc} holds.

The second part of \eqref{thic:trees} follows from the disjointedness of trees
in a directed forest.

\eqref{thic:trees-eq}
The ``single tree condition'' is clearly necessary for $\Tree_1=\Tree_2$.
Assume $\Tree_1\neq \Tree_2$ and fix $v\in V$ such that $\parf_1(v)\neq\parf_2(v)$.
We claim that $[v]_{\Tree_2}$ contains at least two diffrent trees in $\Tree_1$.
Indeed, $\parf_1(v)\in\{v,\parf_2(v)\}\setminus\{\parf_2(v)\}$, hence $v\in\Troot(\Tree_1)$
and $[v]_{\Tree_1}=\Des_{\Tree_1}(v)$ (see \cite[Lemma 2.5 (b)]{pikul1}). 
If $\parf_2(v)\in [v]_{\Tree_1}$, then by \eqref{eq:Des}
there exists $n\in\N$ such that $\parf_1^n(\parf_2(v))=v$.
According to \eqref{eq:thin-anc}, there exists $k\in \N$ such that
$\parf_2^k(\parf_2(v))=v$, which by \eqref{eq:fdef} implies that
$\parf_2(v)=v=\parf_1(v)$, a contradiction. We have proven that
$[v]_{\Tree_1}\neq[\parf_2(v)]_{\Tree_1}$ while both are contained in
$[v]_{\Tree_2}=[\parf_2(v)]_{\Tree_2}$ (by \eqref{thic:trees}).
\end{proof}

Sometimes it is convenient to modify the underlying directed forest.
The following lemma provides a natural sufficient condition under
which the modification does not affect the weighted shift operator.
\begin{Lem}\label{lem:thicklambda}
If $\blambda$ is a system of weights on a directed forest
$\Tree_1=(V,\parf_1)$ and $\Tree_2=(V,\parf_2)$ is a directed forest thicker
than $\Tree_1$, then $\blambda$ is a system of weights on $\Tree_2$ and
the weighted shifts with weights $\blambda$ on directed trees $\Tree_1$ and $\Tree_2$
respectively are equal as operators acting on $\eld{V}$.
\end{Lem}
\begin{proof}
If $v\in\Troot(\Tree_2)$, then by Lemma~\ref{lem:thick2}~\eqref{thic:roots},
$v\in\Troot(\Tree_1)$ and consequently $\lambda_v=0$. This means that $\blambda$ is
a system of weights on $\Tree_2$.

Observe, that by the very definition of $S_\blambda$ neither the domain nor
the values of operator are affected by replacing $\parf_1$ by $\parf_2$.
In both cases we obtain the same function $\varGamma_\blambda$, since
$\parf_1(v)\neq\parf_2(v)$ implies $\lambda_v=0$.
\end{proof}

\begin{Rem}\label{rem:Tlambda-leq}
A benefit from considering disconnected forests is that we can always
``remove the edges with zero weights'' still remaining in the same class.
As stated by \cite[Proposition 3.2]{pikul1}, every weighted shift on $\Tree$
is actually a proper weighted shift on a possibly different
directed forest $\Tree_\blambda$.
 Moreover, $\Tree_\blambda\treeleq\Tree$.
Having Lemmata~\ref{lem:thin-forest} and \ref{lem:thicklambda}, the proof
is straightforward.
This property has no analogue in the class of weighted shifts on directed
trees. Removing edges from a directed tree inevitably disconnects it.\end{Rem}

\section{Operations on directed forests}
In the following section we discuss several important operations on directed
forests. Each of them has its more or less natural counterpart
when considering associated weighted shift operator.
For the sake of the main result of this article the most important operation
is taking the ``leafless support'' of a directed forest. Similarly as
the relation of thickness, this idea is new.

Denote the \emph{direct sum} of the family $\{\Tree_j\}_{j\in J}$ of
(pairwise disjoint) directed forests by $\bigoplus_{j\in J} \Tree_j$
(see \cite[Definition 2.9]{pikul1}).
The family of trees in $\bigoplus_{j\in J} \Tree_j$ is precisely the
union of families of trees in $\Tree_j$ where $j$ varies over $J$.
Any forest can be seen as the direct sum of all its trees, i.e.
\[ \Tree = \bigoplus_{W \text{\,--\,tree in }\Tree} \Tree\rest{W}. \]
Note that if $W$ is a tree, then $\Tree\rest{W}:=(W,\parf_\Tree\rest{W})$
is a directed tree.

An orthogonal sum of weighted shifts on directed forests is equal to
a weighted shift on the direct sum of underlying forests (see \cite[Proposition~3.4]{pikul1}).
The other way round, a weighted shift on a directed forest can be naturally
decomposed into weighted shifts on directed trees. It is formally
described by \cite[Lemma 3.5]{pikul1}.
According to this property, many results regarding weighted shifts on
directed trees immediately extend to operators defined on forests.

\begin{Def}[{\cite[Definition 2.12]{pikul1}}] 
For a directed forest $\Tree=(V,\parf)$ and a positive integer $k$, the
{\bf $k$-th power of $\Tree$} is defined as $\Tree^k:=(V,\parf^{[k]})$,
where
\[ \parf^{[k]}(v):= \begin{cases}
\parf^k(v) &\text{if } \parf^{k-1}(v)\notin \Troot(\Tree)\\
v &\text{if }\parf^{k-1}(v)\in \Troot(\Tree) \end{cases},\qquad v\in V.\]
\end{Def}

A naive way to define ``$k$-th power of $\Tree$'' would be taking $(V,\parf^k)$.
The fact that it is indeed a directed forest and $\Troot((V,\parf^k))=\Troot(\Tree)$
is straightforward from \eqref{eq:fdef}.
Our definition of $\Tree^k$ is motivated by needs of operator theory
(see \cite[Lemma 3.6]{pikul1}) and is natural if we remember
that ``there is no edge from a root to itself''. Thanks to the relation
of thickness and Lemma~\ref{lem:thin-forest} we can immediately
observe that the definition is correct since $\parf^{[k]}\treeleq\parf^k$.

Properties of the powers of directed forests can be found in
\cite[Lemmata 2.13 and 2.14]{pikul1}. We would like to note one
new fact:

\begin{Pro}\label{tp:thickness}
If $\Tree$ and $\tilde\Tree$ are directed forests
such that $\Tree\treeleq\tilde\Tree$,
then $\Tree^k\treeleq\tilde\Tree^k$ for $k\in\N$.
\end{Pro}
\begin{proof}
It follows from the equality $\Chif_{\Tree^k}(v) = \Chit{k}{\Tree}(v)$
(\cite[Lemma 2.13 (c)]{pikul1}),
Lemma~\ref{lem:thick2}~\eqref{thic:childn} and
Theorem~\ref{thm:thickness}~\eqref{thic:child1}.
\end{proof}

According to \cite[Lemma 3.6]{pikul1}, the $k$-th power of a weighted shift on $\Tree$
is a~weighted shift on $\Tree^k$, hence the class of weighted shifts on directed forests
is closed on taking powers of the operator.
If $S_\blambda$ is a~weighted shift on an infinite directed tree, $S_\blambda^k$ is
a~weighted shift on a~directed forest consisting of at least $k$ trees
(see \cite[Lemma~2.14 (c)]{pikul1}).
In particular, for $k\geq 2$, it is no longer a~weighted shift on
a directed tree. The number of trees in the $k$-th power of a rooted directed tree
cannot be (easily) expressed without additional information. See
\cite[Corollary 2.16]{pikul1} for an exemplary estimation in one
particular case.

The next operation will be taking the greatest element among the leafless
directed forests bounded by $\Tree$, with respect to the order of thickness.
\begin{Thm}\label{thm:leaf-supp}
For every directed forest $\Tree$ there exists the
thickest element in the set of all leafless directed forests
thinner than $\Tree$.
\end{Thm}

\begin{proof}
Denote by $\mathcal F$ the family of all functions $\parq:V\to V$
such that $(V,\parq)$ is a leafless directed forest thinner than
$\Tree=(V,\parf)$. Define the function $\mathring\parf:V\to V$ by
\[ \mathring{\parf}(v):=\begin{cases}
v &\text{ if } \parq(v)=v \text{ for every } \parq\in\mathcal F \\
\parf(v) & \text{ if } \parq(v)\neq v \text{ for some } \parq\in\mathcal F
\end{cases},\quad v\in V.\]
We claim that $\mathring{\Tree}:=(V,\mathring{\parf})$ is a directed
forest which is thinner than $\Tree$ and thicker than $(V,\parq)$
for every $\parq\in\mathcal F$.
Indeed, since obviously $\mathring{\parf}\treeleq\parf$, 
Lemma~\ref{lem:thin-forest} justifies that
$\mathring{\Tree}$ is a directed forest and
$\mathring{\Tree} \treeleq\Tree$.
Pick $\parq\in\mathcal F$. Then $(V,\parq)\treeleq \Tree$.
For $v\in V$ we have either $\parq(v)=v$ or $v\neq\parq(v)=\parf(v)$,
and consequently $\mathring\parf(v)=\parf(v)$.
In both cases this means that $\parq(v)\in \{v,\mathring{\parf}(v)\}$, hence
$\mathring{\Tree}$ is indeed thicker than $(V,\parq)$.

It remains to show that $\mathring\Tree$ is leafless.
Let $v\in V\setminus\Troot(\mathring{\Tree})$. By definition of $\mathring\parf$,
it implies that $\mathring{\parf}(v)=\parf(v)$ and there exists $\parq\in\mathcal F$
such that $\parq(v)\neq v$, or equivalently $v$ is not a root of $(V,\parq)$.
By leaflessness of $(V,\parq)$ and
\cite[Lemma~2.4 (f)]{pikul1} 
we get $\Chif_{(V,\parq)}(v)\neq \emptyset$.
Since $(V,\parq)\treeleq \mathring\Tree$,
using Theorem~\ref{thm:thickness}~\eqref{thic:child1} we conclude that also
$\Chif_{\mathring\Tree}(v)\neq \emptyset$, hence $v$ is not a leaf of
$\mathring\Tree$. Since no $v\in V\setminus\Troot(\mathring{\Tree})$ is a leaf,
we have proven the leaflessness of $\mathring{\Tree}$.
\end{proof}
\begin{figure}[hbt]
\includegraphics[scale=1]{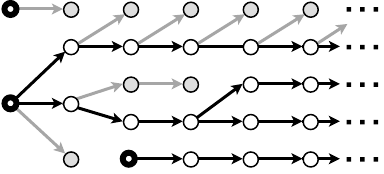}
\caption[An example of a leafless support]{A directed forest (all arrows) and
its leafless support (black arrows).}\label{fig:leaf-support}
\end{figure}
\begin{Def}[leafless support]\label{dfn:leaf-supp}
If $\Tree=(V,\parf)$ is a directed forest, then the thickest
leafless directed forest thinner than $\Tree$ will be denoted by
$\mathring{\Tree}=\bigl(V,\mathring{\parf}\bigr)$
and called the \emph{leafless support of $\Tree$} (see
Figure~\ref{fig:leaf-support} for an illustration).
\end{Def}

The above definition is correct due to Theorem~\ref{thm:leaf-supp}.
The leafless support can be described in the spirit of Corollary~\ref{cor:leafless}
as follows.

\begin{Lem}\label{lem:supp-root}
Let $\Tree=(V,\parf)$ be a directed forest.
Then $v_0\in V$ is not a root of $\mathring\Tree$
if and only if $v_0\notin\Troot(\Tree)$ and there exists a sequence
$\{v_n\}_{n=1}^\infty \subseteq V $
such that $\parf(v_{n+1})= v_n$ for $n\in\ZP$.
\end{Lem}
\begin{proof} Let $v_0\in V\setminus\Troot(\mathring\Tree)$. By Corollary~\ref{cor:leafless}, there exists a sequence
$\{v_n\}_{n=1}^\infty \subseteq V\setminus \Troot(\mathring\Tree)$ such that
$\mathring\parf(v_{n+1})=v_n$ for $n\in\ZP$. Since $\mathring\Tree\treeleq\Tree$,
$\mathring\parf(v_{n+1})\in\{v_{n+1},\parf(v_{n+1})\}$ but
since $v_{n+1}\notin\Troot(\mathring\Tree)$, we have
\[ \parf(v_{n+1})=\mathring\parf(v_{n+1})=v_n,\quad n\in\ZP.\]
It follows from Lemma~\ref{lem:thick2}~\eqref{thic:roots} that
$v_0\notin\Troot(\Tree)\subseteq \Troot(\mathring\Tree)$.

On the other hand, given $\{v_n\}_{n=0}^\infty \subseteq V$
such that $\parf(v_{n+1})=v_n$ for $n\in\ZP$ and $v_0\notin\Troot(\Tree)$,
we can construct a directed forest
$(V,\parq)$, where
\[ \parq(v)=\begin{cases}\parf(v) &\text{ if } v=v_n \text{ for some } n\in\ZP\\
v &\text{ if } v\neq v_n \text{ for all } n\in\ZP\end{cases}.\]
Note that $\parq(v_0)=\parf(v_0)\neq v_0$.
Using Lemma~\ref{lem:thin-forest} one can easily prove that $(V,\parq)$
is indeed a directed forest and $(V,\parq)\treeleq \Tree$. It is evident
that $V=\parq(V)$, which implies that the new forest is leafless,
and as such is thinner than $\mathring\Tree$. Then
\[ v_0\in\Chif_{(V,\parq)}(\parq(v_0))\subseteq \Chif_{\mathring\Tree}(\parf(v_0)),\]
and consequently $v_0$ is not a root of $\mathring\Tree$
(see \cite[Lemma~2.4 (d)]{pikul1}). 
\end{proof}

\begin{Cor}
If $v\in V\setminus\Troot(\mathring\Tree)$, then 
$\mathring\parf^n(v)=\parf^n(v)$ for all $n\in\ZP$. Hence, if two
vertices $v_1,v_2\in V\setminus\Troot(\mathring\Tree)$ belong to the
same tree of $\Tree$, then they are in the same tree of the leafless support.
In particular, assuming $\Tree$ is a directed tree we obtain
that $\mathring\Tree$ contains at most one non-degenerate tree.
\end{Cor}

A naive way to construct ``a leafless support'' would be to
inductively remove all edges leading to leaves. In other words,
we would preserve those children of a vertex which have descendants
in every generation, i.e.\ for every $k\geq1$,
$\Chif[k](v)\neq \emptyset$.
Below we give an example showing that such approach is not sufficient.

\begin{Przyk}\label{ex:no-support}
Consider the directed tree $\Tree=(V,\parf)$, where
\[ V:=\{(m,n)\in\N^2\colon m\geq n\}\cup\{(1,0)\}, \]
\[ \parf\bigl((m,n)\bigr):=\begin{cases}
(m,n-1) & \text{ for } n\geq 2\\
(1,1) & \text{ for } n=1,\ m\geq 2\\
(1,0) & \text{ for } m=1,\ n=0,1\end{cases},\quad (m,n)\in V.\]
See Figure~\ref{fig:no-support}.
Then $\Chif[k]((1,1))\neq \emptyset$ for all $k\in \N$. However,
by Lemma~\ref{lem:supp-root}, $\mathring\Tree=(V,\mathrm{id}_V)$, and in particular $\Chif_{\mathring\Tree}((1,0))=\emptyset$.
\end{Przyk}

\begin{figure}[htb]
\includegraphics[scale=1.1]{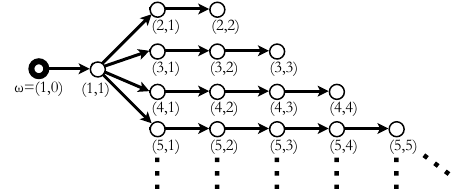}
\caption{The tree from Example~\ref{ex:no-support}.}
\label{fig:no-support}
\end{figure}

One can think of using a transfinite induction for removing
leaves, but we are not going to discuss it further.


\section{Hyponormality of powers} 
We recall that a bounded linear operator $T$ on a (complex) Hilbert space is said to be
\emph{hyponormal} if $[T^*, T]:=T^*T-TT^*\geq 0$. It is
\emph{power hyponormal} if $T^n$ is hyponormal for every $n\geq 1$.
The notion of hyponormality originates from the work of Halmos \cite{halmos}.
We refer the reader to \cite{Con} for the fundamentals of the theory of
hyponormal operators.
\cite[Theorem 5.1.2]{szifty} provides the characterisation of
hyponormality of weighted shifts on directed trees.
The generalisation to the case of arbitrary powers
of proper weighted shifts on directed forests is stated by \cite[Theorem~4.1]{pikul1}.
Characterisation of power hyponormality is summarised in \cite[Corollary~4.2]{pikul1}.

It is well known that for a classical weighted shift hyponormality implies power
hyponormality (it is equivalent to the fact that the sequence of weights is
monotonically increasing in the absolute value).
In \cite{pikul1} there was already proven a characterisation of all those
locally countable leafless directed forests on which every hyponormal
bounded  weighted shift is power hyponormal. These are precisely
the ``forkless forests''. We recall the definition below.

\begin{Def}[{\cite[Definition 4.3]{pikul1}}]\label{dfn:forkless-forest}
We call a directed forest $\Tree=(V,\parf)$ a \emph{forkless forest}
if for every $v\in V^\circ$ its degree equals one.
A forkless forest containing only one tree is called a
\emph{forkless tree} (see Figure~\ref{fig:forkless} for an example).
\end{Def}

\begin{figure}[htb]
\includegraphics[scale=1.15]{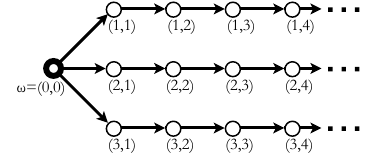}
\caption{A canonical example of a forkless tree.}
\label{fig:forkless}
\end{figure}

Countable forkless trees are characterised by \cite[Proposition 4.5]{pikul1}.
They are\footnote{Except for the rootless forkless tree
$(\mathbb Z,\parf_\mathbb Z)$ mentioned on page~\pageref{dfn:linear-tree}.}
often denoted by $\mathscr{T}_{n,0}$ in the literature
(cf.\ \cite[6.2.10]{szifty}).
Below we recall the previous result on directed forests, all hyponormal
weighted shifts on which are power hyponormal.

\begin{Thm}[{\cite[Theorem 4.8]{pikul1}}]\label{thm:only-stars}
Let $\Tree=(V,\parf)$ be a leafless locally countable directed forest.
Then \NWSR
\begin{abece}
\item every bounded hyponormal weighted shift on $\Tree$ is power hyponormal,
\item if $S_\blambda$ is a bounded hyponormal weighted shift on $\Tree$, then
$S_\blambda^2$ is hyponormal,
\item $\Tree$ is a forkless forest.
\end{abece}
Moreover, the conditions {\rm(a)} and {\rm(b)} remain equivalent
if we restrict to proper weighted shifts.
\end{Thm}

Now we will generalise the result to cover the case of 
directed forests which are not leafless and locally countable.
First note that forkless forests have a neat property involving
the relation of thickness (cf.\ the proof of \cite[Lemma~4.7]{pikul1}).

\begin{Lem}\label{lem:thin-fork}
Every leafless directed forest thinner than a forkless forest is also forkless.
\end{Lem}
\begin{proof} Let $\Tree_2=(V,\parf_2)$ be a forkless forest and
$\Tree_1=(V,\parf_1)$ be a leafless directed forest thinner than $\Tree_2$.
Let $v\in V\setminus \Troot(\Tree_1)$. From the assumption of leaflessness
we get $1\leq \deg_{\Tree_1}(v)$. By Theorem~\ref{thm:thickness}~\eqref{thic:child1}
we see that $\deg_{\Tree_1}(v) \leq \deg_{\Tree_2}(v)$. It follows from
Theorem~\ref{thm:thickness}~\eqref{thic:roots} that $v\notin\Troot(\Tree_2)$,
hence by forklessness $\deg_{\Tree_2}(v)=1$. Combining these inequalities
we obtain $\deg_{\Tree_1}(v)= 1$. This proves that $\Tree_1$ is forkless.
\end{proof}

Using the notion of leafless support (see Definition~\ref{dfn:leaf-supp}),
we can formulate the complete characterisation of directed forests on
which all hyponormal weighted shifts are power hyponormal.

\begin{Thm}\label{thm:star-support}
Let $\Tree=(V,\parf)$ be a directed forest. Then \NWSR
\begin{abece}
\item every bounded hyponormal weighted shift on $\Tree$
is power hyponormal,
\item square of any bounded hyponormal weighted shift on $\Tree$
is hyponormal
\item the leafless support of $\Tree$ is forkless,
\item every leafless forest thinner than $\Tree$ is forkless.
\end{abece}
\end{Thm}
\begin{proof}
Clearly (a)$\Rightarrow$(b).

\pImp{b}{c} Assume, to the contrary, that the leafless support
$\mathring{\Tree}$ is not forkless. We will show that there exists
a non-forkless locally countable leafless directed forest which is thinner
than $\mathring{\Tree}$ and hence $\Tree$.
According to Theorem~\ref{thm:only-stars} and Lemma~\ref{lem:thicklambda}
this would mean that there exists a bounded
hyponormal weighted shift on $\Tree$ with a non-hyponormal square.

To construct such a forest, pick $u,v_1,w_1\in V$ such that
$v_1,w_1\in\Chif_{\mathring{\Tree}}(u)$, $v_1\neq w_1$ and
$\mathring{\parf}(u)\neq u$ ($u$ is the vertex violating  the
forklessness of $\mathring{\Tree}$).
By Corollary~\ref{cor:leafless} applied to $\mathring{\Tree}$ there exist
sequences $\{v_n\}_{n=2}^\infty$ and $\{w_n\}_{n=2}^\infty$ such that
$\mathring{\parf}(v_{n+1})=v_n$ and $\mathring{\parf}(w_{n+1})=w_n$
for $n\in\N$. Then we define a new parent function
\[\parf'(v):=\begin{cases} \mathring{\parf}(v) & \text{ if }v\in W\\
v & \text{ if }v\notin W\end{cases},\quad v\in V,\]
where $W:= \{\mathring\parf^n(u)\colon n\in\ZP\}
\cup\{v_n:n\in\N\}\cup\{w_n:n\in\N\}$.
Then, using Lemma~\ref{lem:thin-forest}, we obtain that
$\Tree'=(V,\parf')$ is indeed a directed forest
($\parf'\treeleq\mathring{\parf}\treeleq\parf$ is evident).
It is immediate from the definition that it is thinner than
$\mathring{\Tree}$.
Note that $\Tree'$ consists exclusively of degenerate trees and
the tree $W$, hence it is locally countable. By the construction of $W$
and $\parf'$ it is leafless ($\parf'(V)=V$). It is not forkless,
since $u\notin\Troot(\Tree')$ and $\Chif_{\Tree'}(u)=\{v_1,w_1\}$.

\pImp{c}{d} follows from Lemma~\ref{lem:thin-fork} and definition of the leafless support.

\pImp{d}{a} Let $S_\blambda$ be a bounded hyponormal weighted shift on $\Tree$.
According to Remark~\ref{rem:Tlambda-leq}
, $S_\blambda$ can be regarded as
a proper weighted shift on $\Tree_\blambda$ and $\Tree_\blambda\treeleq\Tree$.
From \cite[Theorem~4.1]{pikul1} we deduce the leaflessness of $\Tree_\blambda$.
By (d), $\Tree_\blambda$ is forkless.
Finally, \cite[Lemma 4.7]{pikul1} (cf.\ Theorem~\ref{thm:only-stars})
implies power hyponormality of $S_\blambda$.
\end{proof}

\subsection*{Acknowledgements}
Results presented in this paper are mostly part of my PhD dissertation
written under supervision of prof.\ Jan Stochel, for whose guidance
I am very grateful. I would like to thank the reviewers whose
remarks and suggestions had significant impact on the current
shape of the article.


\end{document}